\documentclass[11pt,a4paper]{amsart}
\usepackage[english]{babel}
\usepackage{hyperref}
\usepackage{german}
\usepackage{amsmath}
\usepackage{hyperref}
\usepackage{amssymb}
\usepackage{amsthm}
\usepackage{fontenc}
\usepackage{enumerate}
\usepackage{a4wide}

\theoremstyle{plain}
\newtheorem{satz}{Theorem}[section]

\newtheorem{lem}[satz]{Lemma}

\newtheorem{prop}[satz]{Proposition}

\theoremstyle{definition}
 
\newtheorem{bsp}[satz]{Example}
\newtheorem{frage}[satz]{Question}
\newtheorem*{Frage}{Question}

\theoremstyle{remark}

\newtheorem{bem}[satz]{Remark}

\newcommand{\maxk}[1]{\left\{#1\right\}}
\newcommand{\erz}[1]{\langle#1\rangle}

\DeclareMathOperator{\height}{ht}

\begin{document}

\title[Natural Cohen-Macaulayfication of simplicial affine semigroup rings]{Natural Cohen-Macaulayfication of\linebreak simplicial affine semigroup rings}
\author{Max Joachim Nitsche}
\address{Max\mbox{\;}Planck\mbox{\;}Institute\mbox{\;}for\mbox{\;}Mathematics\mbox{\;}in\mbox{\;}the\mbox{\;}Sciences,\mbox{\;}Inselstrasse\mbox{\;}22,\mbox{\;}04103\mbox{\;}Leipzig,\mbox{\;}Germany}
\email{nitsche@mis.mpg.de}
\thanks{}
\date{\today}
\keywords{Simplicial affine semigroup rings, Cohen-Macaulay property, monomial ideals.}

\subjclass[2010]{Primary 13H10.}

\begin{abstract}

Let $K$ be a field, $B$ a simplicial affine semigroup, and $C(B)$ the corresponding cone. We will present a decomposition of $K[B]$ into a direct sum of certain monomial ideals, which generalizes a construction by Hoa and St\"uckrad. We will use this decomposition to construct a semigroup $\tilde B$ with $B\subseteq \tilde B\subseteq C(B)$ such that $K[\tilde B]$ is Cohen-Macaulay with the property: $\tilde B \subseteq \hat B$ for every affine semigroup $\hat B$ with $B\subseteq \hat B\subseteq C(B)$ such that $K[\hat B]$ is Cohen-Macaulay.

\end{abstract}

\maketitle

\section{Introduction}

By an affine semigroup we mean a finitely generated submonoid of $(\mathbb Z^n,+)$ for some $n\in\mathbb N^+$. Let $K$ be an arbitrary field and $B$ an affine semigroup; as usual $K[B]$ denotes the affine semigroup ring associated to $B$, that is, the $K$-vector space with basis $\{t^b\mid b\in B\}$ and multiplication given by the $K$-bilinear extension of $t^a\cdot t^b=t^{a+b}$. Let $X$ be a subset of $\mathbb Q^n$; we define $C(X):=\{\sum\lambda_ix_i\mid \lambda_i\in\mathbb Q_{\geq0}, x_i\in X\}$ to be the cone spanned by $X$. In the following we will assume that $B$ is a simplicial affine semigroup, that means, by definition, we assume that there are linearly independent elements $e_1,\ldots,e_d\in B$ with $C(\{e_1,\ldots,e_d\})=C(B)$. By $A:=\erz{e_1,\ldots,e_d}$ we denote the submonoid of $B$ generated by $e_1,\ldots,e_d$, thus, $T:=K[A]$ is a polynomial ring in $d$ variables. Note that $\dim T=\dim K[B]=d$. In \cite[Proposition\,2.2]{HSCM}, Hoa and St"uckrad introduced in the homogeneous case a decomposition of $K[B]$ into a direct sum of monomial ideals. Generalizing this result, we will construct certain monomial ideals $I_j$ in $T$ and certain elements $h_j\in C(B)\cap \mathbb Z^n$ such that
\begin{equation}
K[B]\cong \bigoplus\limits_{j=1}^{f} I_j (-h_j)\label{glzerl}
\end{equation}
as $\mathbb Z^n$-graded $T$-modules (see Proposition~\ref{zerl}). Note that the grading on $T$ and on $K[B]$ is always given by $\deg t^a=a$. Some properties of the ring $K[B]$ can be characterized in terms of the semigroup $B$, for example, the Cohen-Macaulay or the Buchsbaum property; for more information we refer to \cite{MHCM, GSWSG, RSHFGA, TBB, CMGM, GSRBB, PSMF, MMTV}. In view of \cite[Theorem\,6.4]{RSHFGA} our decomposition (\ref{glzerl}) can be used to describe the Cohen-Macaulay property, namely, the ring $K[B]$ is Cohen-Macaulay if and only if every ideal $I_j$ is equal to $T$. In Section~\ref{NCM} we will consider the affine semigroup $\tilde B=\erz{e_1,\ldots,e_d,h_1,\ldots,h_f}$ generated by the shifts which occur in the decomposition. We will show in Proposition~\ref{shiftszerl}
$$
K[\tilde B]\cong \bigoplus\limits_{j=1}^{f} T (-h_j)
$$
as $\mathbb Z^n$-graded $T$-modules; meaning $K[\tilde B]$ can be deduced from our decomposition of $K[B]$ replacing $I_j$ by $T$ for all $j=1,\ldots,f$. This shows that the ring $K[\tilde B]$ is always Cohen-Macaulay. Let $B_{sat}$ denote the saturation of $B$ (see Section~\ref{NCM}); by a result of Hochster \cite{MHCM} the ring $K[B_{sat}]$ is Cohen-Macaulay, since the semigroup $B_{sat}$ is normal, see also \cite[Theorem\,6.4]{BGPRKT}. By construction $\tilde B$ and $B_{sat}$ are affine semigroups in $C(B)$ containing $B$ (see Lemma~\ref{contained}), thus, it is natural to ask:

\begin{frage}\label{quest}
Is there a uniquely determined affine semigroup $\hat B$ with $B\subseteq \hat B\subseteq C(B)$ such that $K[\hat B]$ is Cohen-Macaulay, which is minimal among all affine semigroups with these properties?
\end{frage}

This question has a positive answer, more explicitly, $\tilde B$ has exactly this property, see Theorem~\ref{unique}. This implies that $\tilde B$ is always contained in $B_{sat}$, in fact, the semigroup $\tilde B$ could be smaller than $B_{sat}$, see Example~\ref{smaller}.\\

In Section~\ref{decomp} we will introduce the decomposition of $K[B]$, after this we we will study the Cohen-Macaulay property in Section~\ref{NCM}. Finally, we compare our results to the results of Goto, Suzuki, and Watanabe, see Remark~\ref{watanabe}. For unspecified notation we refer to \cite{BGPRKT, EB}.

\section{Decomposition of simplicial affine semigroup rings}\label{decomp}

We define the set 
$$
B_A:=\{x\in B\mid x-a\notin B~\forall a\in A\setminus\{0\}\}.
$$ 
By construction we have if $x\notin B_A$ then $x+y\notin B_A$ for all $x,y\in B$. Moreover, for all $x\in B$ there is an $m\in \mathbb N^+$ such that $mx\in A$, since $C(B)=C(\{e_1,\ldots,e_d\})=C(A)$ by assumption. This shows that $B_A$ is finite. By $G(X)$ we denote the group generated by $X$, for $X\subseteq \mathbb Z^n$. For an element $x\in G(B)$ denote by $\lambda_1^x,\ldots,\lambda_d^x$ the uniquely determined elements of $\mathbb Q$ such that $x=\sum_{i=1}^d\lambda_i^xe_i$. It follows that $x\notin B$ in case that $\lambda_i^x<0$ for some~$i$. Hence for every $x\in B$ we can consider the element $y=x-\sum_{j=1}^dn_je_j\in B$ with $n_j\in\mathbb N$ such that $\sum_{j=1}^dn_j$ is maximal; we get $y\in B_A$. Thus, for all $x\in B$ there is an $y\in B_A$ such that $x=y+\sum_{j=1}^dn_je_j$ for some $n_j\in\mathbb N$. We define $x\sim y$ if $x - y\in G(A)$, hence $\sim$ is an equivalence relation on $G(B)$. Clearly, every element in $G(B)$ is equivalent to an element in $G(B)\cap D$, where
$$
D:=\{x\in \mathbb Q^n \mid x = \sum\nolimits_{i=1}^d \lambda_ie_i, \lambda_i\in \mathbb Q \mbox{ and } 0\leq \lambda_i<1 \mbox{ for all } i=1,\ldots,d\}
$$
and for all $x,y\in G(B)\cap D$ with $x\not=y$ we have $x\not\sim y$, since $e_1,\ldots,e_d$ are linearly independent. Hence the number of equivalence classes $f:=\#(G(B)\cap D)$ in $G(B)$ is finite. Every element in $B$ is by construction equivalent to an element in $B_A$. On the other hand for $x\in G(B)$ we have $x=y-z$ with $y,z\in B$ and again $mz\in A$ for some $m\in\mathbb N^+$. By this we get 
$$
x=y+(m-1)z-mz\sim y+(m-1)z \in B,
$$ 
hence there are exactly $f$ equivalence classes in $B$, $G(B)$, $G(B)\cap D$, and in $B_{A}$. By $\Gamma_1,\ldots,\Gamma_f$ we will denote the equivalence classes on $B_{A}$. We define
$$
h_j:= \sum\nolimits_{i=1}^d \min\maxk{\lambda_i^x\mid x\in\Gamma_j} e_i,
$$
for $j=1,\ldots,f$, hence $h_j\in C(B)$ by construction. Let $i\in\{1,\ldots,d\}$; since $\lambda_i^x-\lambda_i^y\in\mathbb Z$ for all $x,y\in\Gamma_j$, we get for all $x\in\Gamma_j$ that
$$
x-h_j = \sum\nolimits_{i=1}^d\lambda_i^xe_i - \sum\nolimits_{i=1}^d \min\maxk{\lambda_i^y\mid y\in\Gamma_j} e_i=\sum\nolimits_{i=1}^dn_ie_i
$$
for some $n_i\in\mathbb N$, hence $x-h_j\in A$, in particular $x\sim h_j$, and therefore $h_j\in C(B)\cap G(B)$. By construction $\tilde\Gamma_j:=\{t^{x-h_j}\mid x\in\Gamma_j\}$ is a subset of the polynomial ring $T=K[A]$, thus, $I_j:=\tilde\Gamma_jT$ are monomial ideals in $T$ for $j=1,\ldots,f$. In case that $d\geq 2$ we always have $\height I_j\geq 2$ (height) since $\gcd \tilde\Gamma_j=1$ for all $j=1,\ldots,f$. By this we obtain that  all ideals $I_j$ are equal to $T$ in case that $d=1$. In the following we are interested in the canonical $\mathbb Z^n$-grading on $T$ and on $K[B]$, which is given by $\deg t^a=a$. Note that our construction is a generalization of that of \cite[Section\,2]{HSCM} for the homogeneous case; moreover, the next proof is similar as the proof of \cite[Proposition\,2.2]{HSCM}. However, to keep things self contained we will prove it.

\begin{prop}\label{zerl}

There is an isomorphism of $\mathbb Z^n$-graded $T$-modules:
$$K[B]\cong \bigoplus\limits_{j=1}^{f} I_j (-h_j).$$

\end{prop}
\begin{proof}

Define
$$
\psi:  \bigoplus\limits_{j=1}^{f} I_j (-h_j) \rightarrow K[B],
$$
by 
$$
\psi(x_1,\ldots,x_f)=\sum\nolimits_{j=1}^{f} x_j t^{h_j}.
$$
By construction $\psi$ is well defined and preserves the canonical grading. Let $t^x\in K[B]$, that is, $x\in B$. By construction, there is an $y\in B_A$ such that $x=y + \sum_{i=1}^{d}n_i e_i$ for some $n_i\in\mathbb N$. We have $y\in \Gamma_j$ for some $j$, hence $t^{y-h_j}\in I_j$ and therefore
$$
\psi(0,\ldots,0,t^{\sum\nolimits_{i=1}^{d}n_ie_i + y-h_j},0,\ldots,0) = t^{\sum\nolimits_{i=1}^{d}n_ie_i + y-h_j} t^{h_j} = t^x,
$$
since $t^{\sum_{i=1}^{d}n_ie_i}\in T$. This shows that $\psi$ is surjective. Let $x\in \ker \psi$; since $\psi$ is homogeneous we may assume that $x$ is also homogeneous; meaning $x=(\alpha_1t^{c_1},\ldots,\alpha_ft^{c_f})$ for some $\alpha_j\in K$ and some $c_j\in A$, $j=1,\ldots,f$. We get
$$
\psi(x)=\sum\nolimits_{j=1}^{f} \alpha_j t^{c_j+h_j} = 0.
$$
By construction $c_i+h_i\not\sim c_j+h_j$ for all $i\not=j$, hence $c_i+h_i\not= c_j+h_j$ for all $i\not=j$. This shows that $\alpha_j=0$ for all $j=1,\ldots,f$ and therefore $\psi$ is injective.
\end{proof}

\begin{bsp}[{\cite[Example\,10]{GSRBB}}]\label{wbsp}

The following example was given in \cite{GSRBB} to study the relation between the Cohen-Macaulay and the Buchsbaum property. Consider the simplicial affine semigroup $B=\erz{(2, 0), (0, 1), (3, 1), (1, 2)}$, say $A=\erz{(2,0),(0,1)}$. We have 
$$
B_A = \{(0, 0), (3, 1), (1, 2)\}.
$$ 
By this we get $\Gamma_1=\{(0, 0)\}$ and $\Gamma_2=\{(3, 1),(1,2)\}$, thus, $h_1=(0, 0)$, $h_2=(1, 1)$ and therefore $\tilde \Gamma_1=\{1\}$ and $\tilde\Gamma_2=\{t^{(2,0)},t^{(0,1)}\}$. By Proposition~\ref{zerl} it follows that
$$
K[B] \cong T(-(0,0))\; \oplus\; (t^{(2,0)},t^{(0,1)}) T(-(1,1))
$$
as $\mathbb Z^2$-graded $T$-modules.

\end{bsp}

\section{Natural Cohen-Macaulayfication}\label{NCM}

Since $B$ is a simplicial affine semigroup we get that the cone $C(B)$ is pointed, that is, if $x,-x\in C(B)$ it follows that $x=0$. Hence $B$ is a positive affine semigroup, meaning, $0$ is the only unit in $B$, thus, we can fix a positive grading on $K[B]$; see \cite[Page\,58,59]{BGPRKT}. Denote by $T_+:=(t^{e_1},\ldots,t^{e_d})T$ the homogeneous maximal ideal of $T$ and by $H^i_{T_+}(M)$ the $i$-th local cohomology module of a $T$-module $M$ with respect to $T_+$. For a general treatment of the Cohen-Macaulay property and of local cohomology we refer to \cite{BHCMR} and to \cite{BSLC}. The following Theorem is due to Stanley and shows that our canonical decomposition can be used to characterize the Cohen-Macaulay property of $K[B]$:

\begin{satz}[{\cite[Theorem\,6.4]{RSHFGA}}]\label{cmchar}

The following assertions are equivalent:

\begin{enumerate}

\item The ring $K[B]$ is Cohen-Macaulay.

\item There exists $\gamma_1,\ldots,\gamma_f\in B$ such that every element $x\in B$ has a representation of the form $x=\gamma_j+\sum_{i=1}^dn_ie_i$ for some $\gamma_j$ and some  $n_i\in\mathbb N$.

\item We have $\#\Gamma_j=1$ for all $j=1,\ldots,f$.

\item We have $I_j=T$ for all $j=1\ldots,f$.

\end{enumerate}
\end{satz}
\begin{proof}

The equivalence between $(1)$ and $(2)$ was proven in \cite[Theorem\,6.4]{RSHFGA}, provided that $B\subseteq \mathbb N^n$. Since $C(A)=C(B)$ it follows that $t^{e_1},\ldots,t^{e_d}$ is a homogeneous system of parameters of $K[B]$. Thus, $K[B]$ is a Cohen-Macaulay ring if and only if $K[B]$ is a free $T$-module by \cite[Proposition\,6.3]{BGPRKT}. So, $(4)\Rightarrow (1)$ by Proposition~\ref{zerl}. In case that $I_j$ is a proper ideal for some $j$ we get that $\dim T/I_j\leq d-2$, since $\height I_j\geq2$ in this case. We have $H^{i}_{T_+}(T/I_j)\not=0$ for some $i$ with $0\leq i\leq d-2$, and $H^i_{T_+}(T)=0$ for every $i\not=d$, for example, by \cite[Proposition\,A1.16]{TGOS}. From the long exact sequence
$$
\cdots\longrightarrow H^{i}_{T_+}(T)\longrightarrow H^{i}_{T_+}(T/I_j)\longrightarrow H^{i+1}_{T_+}(I_j)\longrightarrow H^{i+1}_{T_+}(T)\longrightarrow \cdots
$$
we obtain $H^{i}_{T_+}(T/I_j)\cong H^{i+1}_{T_+}(I_j)$ for all $i$ with $0\leq i\leq d-2$. Hence $H^i_{T_+}(I_j)\not=0$ for some $i$ with $1\leq i\leq d-1$, and therefore $H^i_{T_+}(K[B])\not=0$ for some $i$ with $1\leq i\leq d-1$ as well by Proposition~\ref{zerl} and the fact that local cohomology commutes with direct sums. Thus, $K[B]$ is not a free $T$-module by a similar argument, and therefore $(1)$ and $(4)$ are equivalent. The assertions $(2)$ and $(3)$ are equivalent as well. Moreover, by construction $(3)\Rightarrow (4)$. In case that $\#\Gamma_j\geq2$ for some $j$ we get for all $x\in\Gamma_j$ an $i\in\{1,\ldots,d\}$ such that $\lambda_i^{h_j}<\lambda_i^x$ and hence $t^{x-h_j}\not=1$. This shows that $I_j$ is a proper monomial ideal in $T$ and we are done.
\end{proof}

Let us consider an affine semigroup $\hat B$ with $B\subseteq\hat B\subseteq C(B)$. The semigroup $\hat B$ is again simplicial, since $C(\hat B)=C(B)=C(\{e_1,\ldots,e_d\})$. In the following we are interested in the semigroup generated by the shifts which occur in the decomposition. We set
$$
\tilde B:=\erz{e_1,\ldots,e_d,h_1,\ldots,h_f},
$$
and we define the saturation $B_{sat}$ of $B$ by $B_{sat}:=C(B)\cap G(B)$. Note that $K[B_{sat}]$ is always Cohen-Macaulay by \cite[Theorem\,1]{MHCM}, since $B_{sat}$ is normal;  see also \cite[Theorem\,6.4]{BGPRKT}.

\begin{lem}\label{contained}

We have $B\subseteq\tilde B\subseteq B_{sat}$.

\end{lem}
\begin{proof}

We get $h_1.\ldots,h_f \in C(B)\cap G(B)$ and therefore $\tilde B\subseteq C(B)\cap G(B) = B_{sat}$. Let $x\in B$; by construction there is an $y\in B_A$ such that  $x=y+\sum_{i=1}^{d} n_i e_i$ for some $n_i\in \mathbb N$. We have $y\in\Gamma_j$ for some $j\in\{1,\ldots,f\}$. Moreover, $y-h_j\in A$ this implies $x=h_j+\sum_{i=1}^{d} n_i' e_i$ for some $n_i'\in\mathbb N$ and therefore $x\in\tilde B$, since $e_i\in\tilde B$.
\end{proof}

This shows that $\tilde B$ is simplicial, since $B\subseteq \tilde B\subseteq C(B)$.

\begin{bsp}\label{smaller}

Let us again consider the affine semigroup $B=\erz{(2, 0), (0, 1), (3, 1), (1, 2)}$ with $A=\erz{(2,0),(0,1)}$. We have $\tilde B=\erz{(2, 0), (0, 1), (1, 1)}$ and $B_{sat}=\mathbb N^2$. Hence $B\subsetneqq \tilde B\subsetneqq B_{sat}$. Moreover, $K[B]$ is not Cohen-Macaulay by Theorem~\ref{cmchar}, but Buchsbaum by \cite{GSRBB}. One can show that $\tilde B_{A}=\{(0,0),(1,1)\}$ and therefore
$$
K[\tilde B] \cong T(-(0,0))\;\oplus\; T(-(1,1))
$$
as $\mathbb Z^2$-graded $T$-modules. It follows that the ring $K[\tilde B]$ is Cohen-Macaulay by Theorem~\ref{cmchar}.

\end{bsp}

In view of the above Example and Proposition~\ref{zerl} it is natural to ask the following:

\begin{Frage}

Is $K[\tilde B]$ always isomorphic to the direct sum of $T$ shifted by $h_j$? 

\end{Frage}

This would imply that $K[\tilde B]$ is always Cohen-Macaulay.  The next Proposition will give a positive answer to this question:

\begin{prop}\label{shiftszerl}

We have 
$$
K[\tilde B] \cong \bigoplus_{j=1}^f T(-h_j)
$$
as $\mathbb Z^n$-graded $T$-modules.

\end{prop}
\begin{proof}
By the comments in the beginning of Section~\ref{decomp} we get that the number of equivalence classes on $\tilde B_{A}$ is equal to that on $B_A$, since $G(B)=G(\tilde B)$. By construction we have $h_k\not\sim h_l$ for all $k,l\in\{1,\ldots,f\}$ with $k\not=l$, hence for all $x'\in\tilde B$ there is a $j\in\{1,\ldots,f\}$ such that $x'\sim h_j$. We will show that if $x'\sim h_j$ for some $j\in\{1,\ldots,f\}$, then $x'-h_j\in A$, that is, $x'=h_j+\sum_{i=1}^dn_ie_i$ for some $n_i\in\mathbb N$ and therefore $\Gamma_1'=\{h_1\},\ldots,\Gamma_f'=\{h_f\}$ are the equivalence classes on $\tilde B_{A}$ and we are done by Proposition~\ref{zerl} and construction. Let $x'\in\tilde B$, that is,
$$
x'=\sum\nolimits_{t=1}^{d} n_t' e_t + \underbrace{\sum\nolimits_{t=1}^{f} n_th_t}_{= x} = \sum\nolimits_{t=1}^{d} n_t' e_t + x,
$$
for some $n_t',n_t\in\mathbb N$. We show by induction over $n:=\sum_{t=1}^{f}n_t$ that for $h_j$ with $h_j\sim x$ for some $j\in\{1,\ldots,f\}$ we have $x-h_j\in A$ and therefore $h_j\sim x'$ and $x'-h_j\in A$ as well.\\ 
\underline{$n=0$}: This is clear, since $h_j=0$ for some $j\in\{1,\ldots,f\}$.\\ 
\underline{$n>0$}: We have $x=\sum_{t=1}^{f} n_th_t=x''+h_i$ for some $i\in\{1,\ldots,f\}$, by induction there is a $h_j$ for some $j\in\{1,\ldots,f\}$ such that $h_j\sim x''$ and $x''-h_j\in A$. We have $x\sim h_l$ for some $l\in\{1,\ldots,f\}$. It is now sufficient to show that $h_j+h_i-h_l\in A$; this implies $x-h_l\in A$, since $x''-h_j\in A$. Note that $h_l\sim x'' +h_i\sim h_j+h_i$. We show that for $k=1,\ldots,d$ we have $\lambda_k^{h_{l}}\leq\lambda_k^{h_j}+\lambda_k^{h_i}$. This implies 
$$
h_j+h_i-h_l=\sum\nolimits_{t=1}^d\underbrace{(\lambda_t^{h_j}+\lambda_t^{h_i}-\lambda_t^{h_{l}})}_{\geq 0}e_t\in A,
$$
since $\lambda_k^{h_j}+\lambda_k^{h_i}-\lambda_k^{h_{l}}\in\mathbb Z$ for all $k=1,\ldots,d$. Let $\Gamma_1,\ldots,\Gamma_f$ be the equivalence classes on $B_A$. Fix one $k\in\{1,\ldots,d\}$. By construction there is an element $y_{jk}\in\Gamma_j$ with $\lambda_k^{y_{jk}}=\lambda_k^{h_{j}}$ and some $y_{ik}\in\Gamma_i$ with $\lambda_k^{y_{ik}}=\lambda_k^{h_{i}}$. Note that $y_{jk}\sim h_j$ and $y_{ik}\sim h_i$, hence 
\begin{equation}y_{jk}+y_{ik}\sim h_j+h_i\sim h_l.\label{tilde}\end{equation}
We have $y_{jk}+y_{ik}\in B$ and therefore there is an $s\in B_A$ such that:
$$
y_{jk}+y_{ik}=s+\sum\nolimits_{t=1}^{d}n_t e_t,
$$
for some $n_t\in\mathbb N$. Clearly, $\lambda_k^s\leq \lambda_k^{y_{jk}}+\lambda_k^{y_{ik}}=\lambda_k^{h_j}+\lambda_k^{h_i}$. We have $h_l\stackrel{(\ref{tilde})}{\sim} y_{jk}+y_{ik}\sim s$ and therefore $\lambda_k^{h_{l}}\leq \lambda_k^s$, since $s\in\Gamma_l$. This implies $\lambda_k^{h_{l}}\leq\lambda_k^{h_j}+\lambda_k^{h_i}$ as required.
\end{proof}

That means that $K[\tilde B]$ can be deduced from the decomposition of $K[B]$ in Proposition~\ref{zerl} replacing $I_j$ by $T$ for all $j=1,\ldots,f$. We will now give an answer to Question~\ref{quest} raised in the introduction.

\begin{satz}\label{unique}

Let $B$ be a simplicial affine semigroup, and $\tilde B$ be as above. The ring $K[\tilde B]$ is Cohen-Macaulay for the affine semigroup $\tilde B$ with $B\subseteq\tilde B\subseteq C(B)$, moreover, if $\hat B$ is an affine semigroup with $B\subseteq \hat B\subseteq C(B)$ such that the ring $K[\hat B]$ is Cohen-Macaulay, then $\tilde B\subseteq \hat B$.

\end{satz}
\begin{proof}

By Proposition~\ref{shiftszerl} and Theorem~\ref{cmchar} we get that the ring $K[\tilde B]$ is Cohen-Macaulay, moreover, by Lemma~\ref{contained} $B\subseteq\tilde B\subseteq C(B)$. Let $\hat B$ be an affine semigroup with $B\subseteq\hat B\subseteq C(B)$ such that $K[\hat B]$ is Cohen-Macaulay; again $\hat B$ is simplicial. We show that $h_j\in \hat B$ for all $j=1,\ldots,f$ and therefore $\tilde B\subseteq\hat B$, since $e_1,\ldots,e_d\in \hat B$. By Theorem~\ref{cmchar} we know that $\hat\Gamma_1=\{\hat h_1\},\ldots,\hat\Gamma_{\hat f}=\{\hat h_{\hat f}\}$ are the equivalence classes on $\hat B_{A}$. We have 
$$
f=\#(D\cap G(B))\leq \#(D\cap G(\hat B))=\hat f,
$$
since $G(B)\subseteq G(\hat B)$. For all $j\in\{1,\ldots,f\}$ there is an $x\in B$ and an $i\in\{1,\ldots,\hat f\}$ such that $x\sim h_j$ and $x\sim\hat h_i$, that is, we may assume $h_j\sim\hat h_j$ for $j=1,\ldots,f$. Fix one $j\in\{1,\ldots,f\}$, and let $k\in\{1,\ldots,d\}$; we will show that $\lambda_k^{\hat h_{j}}\leq\lambda_k^{h_j}$. This implies $\lambda_k^{h_j}-\lambda_k^{\hat h_{j}}\in\mathbb N$, since $\lambda_k^{h_j}-\lambda_k^{\hat h_{j}}\in\mathbb Z$. Thus, $h_j-\hat h_j\in A$ and therefore $h_j\in \hat B$, since $\hat h_j\in \hat B$ and $A\subseteq \hat B$. There is an element $x\in B$ with $x\sim h_j$ and $\lambda_k^x=\lambda_k^{h_j}$. Since $x\in \hat B$, $x\sim\hat h_j$, and $\hat\Gamma_j=\{\hat h_j\}$ we get $x=\hat h_j+\sum_{i=1}^dn_ie_i$ for some $n_i\in\mathbb N$ and therefore $\lambda_k^{\hat h_{j}}\leq\lambda_k^x=\lambda_k^{h_j}$.
\end{proof}

\begin{bem}

There is an exact sequence of $\mathbb Z^n$-graded $T$-modules:
$$
0\longrightarrow K[B] \longrightarrow K[\tilde B] \longrightarrow K[\tilde B\setminus B] \longrightarrow 0.
$$
By Proposition~\ref{zerl} and Proposition~\ref{shiftszerl} we have
$$
K[\tilde B\setminus B] \cong \bigoplus\limits_{j=1}^{f} T / I_j (-h_j)
$$
as $\mathbb Z^n$-graded $T$-modules. Hence $\dim K[\tilde B\setminus B] \leq \dim K[B]-2$, since $\height I_j\geq2$; provided that $d\geq2$.

\end{bem}

\begin{bem}\label{watanabe}

Assume that $B\subseteq \mathbb N^n$ for some $n\in\mathbb N^+$. Let $B'$ be the extension of $B$ in $C(B)$ studied by Goto, Suzuki, and Watanabe in \cite{GSWSG}, or by Hoa and Trung in \cite{CMGM} (see also \cite{MMTV, PSMF}). They proved that $B=B'$ if and only if $K[B]$ is a Cohen-Macaulay ring. Since $B'=(B')'$, we get that $K[B']$ is Cohen-Macaulay, hence $\tilde B\subseteq B'$ by Theorem~\ref{unique}. Conversely, $\tilde B=\tilde B'$, since $K[\tilde B]$ is Cohen-Macaulay.  We have $B\subseteq\tilde B$, hence $B'\subseteq\tilde B'=\tilde B$ and therefore $\tilde B=B'$.

\end{bem}

\section*{Acknowledgement}

The author would like to thank J\"urgen St\"uckrad for many helpful discussions.

\end{document}